\documentclass{amsart}
\usepackage{amsmath, amsthm, amsfonts, amssymb, txfonts, hyperref}
\newtheorem{thrm}{Theorem}[section]
\newtheorem{lmm}[thrm]{Lemma}
\newtheorem{prpstn}[thrm]{Proposition}

\newtheorem*{cnjctr}{Conjecture}
\numberwithin{equation}{section}
\newtheoremstyle{named}{}{}{\itshape}{}{\bfseries}{.}{.5em}{\thmnote{#3}}
\theoremstyle{named}

\begin{document}
\title{Bounded length intervals containing two primes and an almost-prime II}
\author{James Maynard}
\address{Mathematical Institute, 24–-29 St Giles', Oxford, OX1 3LB}
\email{maynard@maths.ox.ac.uk}
\thanks{Supported by EPSRC Doctoral Training Grant EP/P505216/1 }
\date{}
\begin{abstract}
Zhang has shown there are infinitely many intervals of bounded length containing two primes. It appears that the current techniques cannot prove that there are infinitely many intervals of bounded length containing three primes, even if strong conjectures such as the Elliott-Halberstam conjecture are assumed. We show that there are infinitely many intervals of length at most $10^8$ which contain two primes and a number with at most 31 prime factors.
\end{abstract}
\maketitle
\section{Introduction}
We are interested in trying to understand how small gaps between primes can be. If we let $p_n$ denote the $n^{\text{th}}$ prime, it is conjectured that
\begin{equation}
\liminf_n (p_{n+1}-p_n)=2.
\end{equation}
This is the famous twin prime conjecture. More generally, we can look at the difference $p_{n+k}-p_n$. It would follow from the Hardy-Littlewood prime $k$-tuples conjecture that
\begin{equation}
\liminf_n(p_{n+k}-p_n)\ll k\log{k}.
\end{equation}
In particular, we expect that $\liminf_{n}(p_{n+1}-p_n)$ is finite for each $k$.

For $k=1$ the recent breakthrough of Zhang \cite{Zhang} has shown unconditionally that
\begin{equation}
\liminf_n (p_{n+1}-p_n)<7\cdot 10^7.
\end{equation}
For $k>1$ we have much less precise knowledge. The best results are due to Goldston, Pintz and Y\i ld\i r\i m \cite{GPYIII}, who have shown
\begin{equation}
\liminf_n\frac{p_{n+k}-p_n}{\log{p_n}}<e^\gamma (\sqrt{k}-1)^2.
\end{equation}
In particular, we do not know whether $\liminf(p_{n+k}-p_n)$ is finite when $k>1$.

Both unconditional results are based on the `GPY method' for showing the existence of small gaps between primes. This method relies heavily on results about primes in arithmetic progressions. We say that the primes have `level of distribution' $\theta$ if, for any constant $A$, there is a constant $C=C(A)$ such that
\begin{equation}
\sum_{q\le x^\theta(\log{x})^{-C}}\max_{\substack{a\\(a,q)=1}}\Biggl|\sum_{\substack{p\equiv a \pmod{q}\\ p\le x}}1-\frac{\text{Li}(x)}{\phi(q)}\Biggr|\ll_A\frac{x}{(\log{x})^A}.
\end{equation}
The Bombieri-Vinogradov theorem states that the primes have level of distribution $1/2$, and the major ingredient in Zhang's proof that $\liminf(p_{n+1}-p_n)$ is finite is a slightly weakened version of the statement that the primes have level of distribution $1/2+1/584$.

It is believed that further improvements in the level of distribution of the primes are possible, and Elliott and Halberstam \cite{ElliottHalberstam} conjectured the following much stronger result.
\begin{cnjctr}[Elliott-Halberstam Conjecture]
For any fixed $\epsilon>0$, the primes have level of distribution $1-\epsilon$.
\end{cnjctr}
Friedlander and Granville \cite{FriedlanderGranville} have shown that the primes do not have level of distribution $1$, and so the Elliott-Halberstam conjecture represents the strongest possible result of this type.

Under the Elliott-Halberstam conjecture the GPY method \cite{GPYI} shows that for $k=1$
\begin{equation}
\liminf_n (p_{n+1}-p_n)\le 16.
\end{equation}
If we consider $k>1$, however, we are unable to prove such strong results, even under the full strength of the Elliott-Halberstam conjecture. In particular we are unable to prove that there are infinitely many intervals of bounded length that contain at least 3 primes. The GPY methods can still be used, but even with the Elliott-Halberstam conjecture we are only able to prove that
\begin{equation}
\liminf_n \frac{p_{n+2}-p_n}{\log{p_n}}=0.
\end{equation}
Therefore it  appears that we are unable to show that $\liminf(p_{n+2}-p_n)$ is finite with the current methods. As an approximation to the conjecture, it is common to look for \textit{almost-prime} numbers instead of primes, where almost-prime indicates that the number has only a `few' prime factors.

Earlier work of the author \cite{Maynard} has shown that, assuming a generalization of the Elliott-Halberstam conjecture for numbers with at most 4 prime factors, there are infinitely many intervals of bounded length containing two primes and a number with at most 4 prime factors.

Pintz \cite{PintzII} has shown that Zhang's result can be extended to show that there are infintely many intervals of bounded length which contain two primes and a number with at most $O(1)$ prime factors. Pintz doesn't give an explicit bound on the number of prime factors for the almost-prime.

We extend this work to show that there are infinitely many intervals of bounded length which contain two primes and a number with at most 31 prime factors.
\section{Main result}
\begin{thrm}\label{thrm:MainTheorem}
There are infinitely many integers $n$ such that the interval $[n,n+10^8]$ contains two primes and a number with at most 31 prime factors.
\end{thrm}
Our result is naturally based heavily on the work of Zhang \cite{Zhang}, and on the GPY method. We follow a similar method to the author's earlier paper \cite{Maynard}, but to simplify the argument we detect numbers with at most $r$ prime factors by using terms weighted by the divisor function. To estimate these terms we rely on earlier work of Ho and Tsang \cite{HoTsang}.
\section{Proof of Theorem \ref{thrm:MainTheorem}}
We let $L^{(1)}_i(n)=n+h_i$ ($1\le i\le k$) be distinct linear functions with integer coefficients. Moreover, we assume that the product function $\Pi^{(1)}(n)=\prod_{i=1}^{k} L_i(n)$ has no fixed prime divisor. We adopt a normalization of our functions, due to Heath-Brown \cite{HB}. We let $L_i(n)=L_i^{(1)}(An+a_0)=An+b_i$ where the constants $A,a_0>0$ are chosen such that for all primes $p$ we have
\begin{equation}\#\{1\le a\le p:\prod_{i=1}^{k}L_i(n)\equiv 0\pmod{p}\}=\begin{cases}k,\qquad&p\nmid A,\\ 0,&p|A.\end{cases}\end{equation}
We now set $\Pi(n)=\prod_{i=1}^{k}L_i(n)$.

We consider the sum
\begin{align}
S=S(B)=\sum_{\substack{N\le n<2N\\ \Pi(n)\text{ square-free}\\ n\equiv a\pmod{A}}}\Biggl(\sum_{i=1}^{k-1}\theta(L_i(n))-1-\frac{\tau(L_k(n))}{B}\Biggr)\Biggl(\sum_{d|\Pi(n)}\lambda_d\Biggr)^2,
\end{align}
where $\tau$ denotes the divisor function, $\theta$ is the function defined by
\begin{align}
\theta(n)&=\begin{cases}1,\qquad &n\text{ is prime},\\ 0,&\text{otherwise,}\end{cases}
\end{align}
and the $\lambda_d$ are real constants (to be chosen later).

We wish to show, for a suitable choice of positive constants $B$ and $k$, that $S>0$ for any large $N$. If $S>0$ for some $N$, then at least one term in the sum over $n$ must have a strictly positive contribution. Since the $\lambda_d$ are all reals, we see that if there is a positive contribution from $n\in[N,2N)$, then one of the following must hold.
\begin{enumerate}
\item At least three of the $(L_i(n))_{i=1}^{k-1}$ are primes.
\item At least two of the $(L_i(n))_{i=1}^{k-1}$ are primes, and $\tau(L_k(n))<B$.
\end{enumerate}
Therefore, in either case we must have at least two of the $(L_i(n))_{i=1}^{k-1}$ prime and one other integer with at most $\lfloor\log_2{B}\rfloor$ prime factors. Since this holds for all large $N$, we see there must be infinitely many integers $n$ such that two of the $L_i^{(1)}(n)$ are prime and one other of the $L_i^{(1)}(n)$ has at most $\lfloor\log_2{B}\rfloor$ prime factors.

We first remove the condition that $\Pi(n)$ be square-free in the sum over $n$, and then we split $S$ up into separate terms which we will estimate individually.
\begin{align}
S&\ge \sum_{N\le n<2N}\Biggl(\sum_{i=1}^{k-1}\theta(L_i(n))-1-\frac{\tau(L_k(n))}{B}\Biggl)\Biggl(\sum_{d|\Pi(n}\lambda_d\Biggr)^2-kS'\nonumber\\
&=-S_1+\sum_{i=1}^{k-1}S_2(L_i)-\frac{1}{B}S_3(p)-kS',\label{eq:SExpansion}
\end{align}
where
\begin{align}
S'&=\sum_{\substack{N\le n< 2N\\ \Pi\text{ not square-free}}}\Bigl(\sum_{d|\Pi(n)}\lambda_d\Bigr)^2,\\
S_1&=\sum_{\substack{N\le n <2N\\ n\equiv a\pmod{A}}}\Bigl(\sum_{d|\Pi(n)}\lambda_d\Bigr)^2,\\
S_2(L_i)&=\sum_{\substack{N\le n <2N\\ n\equiv a\pmod{A}}}\theta(L_i(n))\Bigl(\sum_{d|\Pi(n)}\lambda_d\Bigr)^2,\\
S_3&=\sum_{\substack{N\le n <2N\\ n\equiv a\pmod{A}}}\tau(L_k(n))\Bigl(\sum_{d|\Pi(n)}\lambda_d\Bigr)^2.
\end{align}
We will use the following proposition to estimate the terms above.
\begin{prpstn}\label{prpstn:MainProp}
Let $\varpi=1/1168$ and $D=N^{1/4+\varpi}/A$. Let $D_1=N^\varpi/A$ and $\mathcal{P}=\prod_{p\le D_1}p$. For $d<D$ with $d|\mathcal{P}$ we let
\[\lambda_d=\frac{\mu(d)}{(k+l)!}\left(\log\frac{D}{d}\right)^{k+l},\]
and let $\lambda_d=0$ otherwise. Then we have
\begin{align*}
S'&=o(N(\log{N})^{k+2l}),\\
S_1&\le \frac{\mathfrak{S}N(\log{D})^{k+2l}}{(k+2l)!}\binom{2l}{l}\left(1+\kappa_1+o(1)\right)),\\
S_2(L_i)&\ge \frac{\mathfrak{S}N(\log{D})^{k+2l+1}}{(k+2l+1)!\log{N}}\binom{2l+2}{l+1}\left(1-\kappa_2+o(1)\right),\\
S_3&\le \frac{\mathfrak{S}N(\log{D})^{k+2l}}{(k+2l-1)!}\binom{2l-2}{l-1}\left(\frac{6l-4}{l(k+2l)}+\frac{\log{N}}{\log{D}}+\kappa_3\left(6\varpi+\frac{\log{N}}{\log{D}}\right)+o(1)\right).
\end{align*}
where
{\allowdisplaybreaks
\begin{align*}
\kappa_1&=\delta_1(1+\delta_2^2+(\log{293})k)\binom{k+2l}{k},\\
\kappa_2&=\delta_1(1+\delta_2^2+(\log{293})k)\binom{k+2l+1}{k-1},\\
\kappa_3&=\delta_1(1+\delta_2^2+(\log{293})(k+1))\binom{k+2l-1}{k+1},\\
\delta_1&=(1+4\varpi)^{-k},\\
\delta_2&=1+\sum_{\nu=1}^{293}\frac{(\log{293})k^\nu}{\nu!},\\
\mathfrak{S}&=\prod_{p|A}\left(1-\frac{1}{p}\right)^{-k}\prod_{p\nmid A}\left(1-\frac{k}{p}\right)\left(1-\frac{1}{p}\right)^{-k}.
\end{align*}}
\end{prpstn}
We can now establish our main theorem using Proposition \ref{prpstn:MainProp}. Substituting the bounds into \eqref{eq:SExpansion} we obtain
\begin{equation}S\ge \frac{N\mathfrak{S}(\log{D})^{k+2l}}{(k+2l)!}\binom{2l}{l}\left(\frac{(k-1)(2l+1)(1+4\varpi)(1-\kappa_2)}{(k+2l+1)(2l+2)}-1-\kappa_1-\frac{c_0}{B}+o(1)\right),\end{equation}
where
\begin{equation}c_0=\frac{l(k+2l)}{4l-2}\left(\frac{6l-4}{l(k+2l)}+\frac{\log{N}}{\log{D}}+\kappa_3\left(6\varpi+\frac{\log{N}}{\log{D}}\right)\right)+o(1).\end{equation}
We now choose $k=4.5\times 10^6$, $l=300$. By a simple computation analogous to that giving \cite[inequality (4.21)]{Zhang} we certainly have
\begin{equation}
\kappa_1,\kappa_2,\kappa_3\le \exp(-1000).
\end{equation}
Thus, by computation, we see that for $N$ sufficiently large we have
\begin{align}
\frac{(k-1)(2l+1)(1+4\varpi)(1-\kappa_2)}{(k+2l+1)(2l+2)}-1-\kappa_1&\ge \left(1-\frac{1}{600}-\frac{602}{4500000}\right)\frac{1172}{1168}-1-3e^{-100}\nonumber\\
&\ge 0.0016\end{align}
and
\begin{equation}c_0\le k+2l+2\le 460000.\end{equation}
We now choose $B=2^{32}-1\ge 4000000000$, and we see that
\begin{align}
S&\ge \frac{N\mathfrak{S}(\log{D})^{k}}{(k+2l)!}\binom{2l}{l}\left(0.0016-\frac{4600000}{4000000000}\right)\nonumber\\
&\ge 0.00045\frac{N\mathfrak{S}(\log{D})^{k}}{(k+2l)!}\binom{2l}{l}.\end{align}
Thus for any admissible $k$-tuple of linear functions has infinitely many integers $n$ for which two of the functions are prime and $n$, and another function has at most 31 prime factors. A computation now reveals that
\begin{equation}
\pi(10^8)-\pi(4.5\times10^6)\ge 4.5\times 10^6.
\end{equation}
Therefore we can form an admissible $k$-tuple of linear functions of the form $L_i(n)=n+h_i$ with $k=4.5\times 10^6$ and $0\le h_i\le 10^8$, by letting $h_i=p_{m+i}-p_{m+1}$ where $p_m$ is the largest prime smaller than $4.5\times 10^6$. This shows that there are infinitely many intervals of length at most $10^8$ which contain two primes and a number with at most 31 prime factors.

We comment here that with slightly more care one can take $\kappa_1$, $\kappa_2$, and $\kappa_3$ to be rather smaller than the expressions given in Proposition \ref{prpstn:MainProp}. This allows us to show that $S>0$ for smaller values of $k$, which in turn allows us to reduce the number of prime factors required for the almost-prime from 31 to 29. Moreover, any improvement in the constant $\varpi$ occurring in Zhang's paper would give a corresponding improvement here. By choosing $k$ and $l$ optimally, we would have that there are infinitely intervals of bounded length containing two primes and one almost-prime with $\approx 3\log_2{\frac{3}{4\varpi}}$ prime factors.
\section{Lemmas}
The proof of the bounds for the sums $S'$, $S_1$ and $S_2$ essentially already exists in the literature. Ho and Tsang \cite{HoTsang} evaluate a sum very similar to $S_3$, but in their case the $\lambda_d$ are non-zero on square-free $d<D$ for which $d|\mathcal{P}$. We therefore require some estimates to show that the error in replacing our sieve weights by the ones used by Ho and Tsang is small, analogously to \cite[Sections \S4 and \S5]{Zhang}. Our work naturally relies heavily on the papers \cite{Zhang} and \cite{HoTsang}, and is far from self-contained. 

We recall the definitions of $D,D_1,\mathcal{P},\varpi,\lambda_d$ and $\mathfrak{S}$ from Proposition \ref{prpstn:MainProp}. As in \cite{Zhang}, we also define the quantity $D_0=(\log{D})^{1/k}$.
\begin{lmm}\label{lmm:DivisorBasic}
Let $\varrho_3$ be the multiplicative function supported on square-free integers coprime to $A$ satisfying $\varrho_3(p)=k+1-k^2/p$ for $p\nmid A$. Then
\[\sum_{N\le n<2N}\tau(L_k(n))\left(\sum_{d|\Pi(n)}\lambda_d\right)^2=\frac{N\phi(A)}{A}\left((\log{N}+O(1))M_1-2M_2+M_3\right)+o(N(\log{N})^{k+2l}),\]
where
\begin{align*}
M_1&=\sum_{d,e|\mathcal{P}}\frac{\lambda_d\lambda_e\varrho_3([d,e])}{[d,e]},\\
M_2&=\sum_{p\nmid A}\frac{2(p-k)\log{p}}{(k+1)p-k}\sum_{\substack{d,e|\mathcal{P}\\p|d}}\frac{\lambda_d\lambda_e\varrho_3([d,e])}{[d,e]},\\
M_3&=\sum_{p\nmid A}\frac{2(p-k)\log{p}}{(k+1)p-k}\sum_{\substack{d,e|\mathcal{P}\\p|d,e}}\frac{\lambda_d\lambda_e\varrho_3([d,e])}{[d,e]}.
\end{align*}
\end{lmm}
\begin{proof}
This follows from the argument of \cite[Pages 254-255]{HB}, with changes only to the notation. We note that the $\lambda_d$ are supported on $d<D<N^{1/3-\epsilon}$, as required for the argument. The error term is larger since our $\lambda_d$ are large by a factor $(\log{D})^{k+l}/(k+l)!$.
\end{proof}
\begin{lmm}\label{lmm:Rho3Estimates}
Let $\varrho_3$ be as defined in Lemma \ref{lmm:DivisorBasic}, and let
\begin{align*}
g(y)&=\begin{cases}\frac{1}{(k+l)!}\left(\log{\frac{D}{y}}\right)^{k+l},\qquad &y<D,\\ 0,&\text{otherwise,}\end{cases}\\
\mathcal{A}_3(d)&=\sum_{(r,d)=1}\frac{\mu(r)\varrho_3(r)}{r}g(dr),\\
\theta_3(d)&=\prod_{p|d}\left(1-\frac{\varrho_3(p)}{p}\right)^{-1}.
\end{align*}
Then if $d<D$ is square-free we have
\begin{align*}
\mathcal{A}_3(d)&=\frac{\theta_3(d)}{(l-1)!}\mathfrak{S}\frac{A}{\phi(A)}\left(\log{\frac{D}{d}}\right)^{l-1}+O\left((\log{D})^{l-2+\epsilon}\right),\\
\sum_{d\le x^{1/4}}\frac{\varrho_3(d)\theta_3(d)}{d}&=\frac{(1+4\varpi)^{-k-1}}{(k+1)!}\mathfrak{S}^{-1}\frac{\phi(A)}{A}(\log{D})^{k+1}+O((\log{D})^{k-1})
\end{align*}
\end{lmm}
\begin{proof}
The proof is entirely analogous to that of \cite[Lemmas 3 and 4]{Zhang}, the only difference being we have $\varrho_3$, $\mathfrak{S}A/\phi(A)$, $k+1$ and $l-1$ in place of $\varrho_1$, $\mathfrak{S}$, $k_0$ and $l_0$ in the argument.
\end{proof}
\begin{lmm}\label{lmm:EasyBound}
Let
\[M_1^*=\sum_{d,e}\frac{\lambda_d\lambda_e\varrho_3([d,e])}{[d,e]}.\]
Then we have that
\[|M_1-M_1^*|\le\kappa_3\frac{A}{\phi(A)}\binom{2l-2}{l-1}\frac{\mathfrak{S}(\log{D})^{k+2l-1}}{(k+2l-1)!}(1+o(1)),\]
where
\begin{align*}
\kappa_3&=\delta_1^{(3)}(1+(\delta_2^{(3)})^2+(\log{293})(k+1))\binom{k+2l-1}{k+1},\\
\delta_1&=(1+4\varpi)^{-k},\\
\delta_2&=1+\sum_{\nu=1}^{293}\frac{(\log{293})k^\nu}{\nu!}.
\end{align*}
\end{lmm}
\begin{proof}
The proof is entirely analogous to \S 4 of \cite{Zhang}, using Lemma \ref{lmm:Rho3Estimates} in place of \cite[Lemma 2 and Lemma 3]{Zhang} and replacing $\varrho_1$, $\mathfrak{S}$, $k_0$ and $l_0$ with $\varrho_3$, $\mathfrak{S}A/\phi(A)$, $k+1$ and $l-1$ in the relevant places.
\end{proof}
\begin{lmm}\label{lmm:EasypBound}
Let 
\[M_3^*=\sum_{\substack{p\nmid A\\ p\le D_1}}\frac{2(p-k)\log{p}}{(k+1)p-k}\sum_{p|d,e}\frac{\lambda_d\lambda_e\varrho_3([d,e])}{[d,e]}.\]
Then we have
\[|M_3-M_3^*|\le 2\varpi\kappa_3\frac{A}{\phi(A)}\binom{2l-2}{l-1}\frac{\mathfrak{S}(\log{D})^{k+2l}}{(k+2l-1)!}(1+o(1)),\]
where $\kappa_3$ is defined in Lemma \ref{lmm:EasyBound}.
\end{lmm}
\begin{proof}
We first fix $p$ and consider the difference in the inner sums over $d$ and $e$. This inner sum can be evaluated by essentially the same argument as section \S4 of \cite{Zhang}. The condition $p|d,e$ corresponds to $p|(d,e)$, which in the notation of  \cite[section \S4]{Zhang} introduces the condition $p|d_0$. Writing $d_0$ in place of $d_0/p$ then gives in place of the sums $\Sigma_1,\Sigma_2$ and $\Sigma_3$ the sums
\begin{align*}
\Sigma_{1,p}&=\sum_{d_0\le x^{1/4}/p}\sum_{d_1}\sum_{d_2}\frac{\mu(d_1d_2)\varrho_3(pd_0d_1d_2)}{pd_0d_1d_2}g(pd_0d_1)g(pd_0d_1),\\
\Sigma_{2,p}&=\sum_{\substack{d_0\le x^{1/4}/p\\ d_0|\mathcal{P}}}\sum_{d_1|\mathcal{P}}\sum_{d_2|\mathcal{P}}\frac{\mu(d_1d_2)\varrho_3(pd_0d_1d_2)}{pd_0d_1d_2}g(pd_0d_1)g(pd_0d_1),\\
\Sigma_{3,p}&=\sum_{\substack{x^{1/4}/p<d_0\le D/p\\ d_0\nmid\mathcal{P}}}\sum_{d_1}\sum_{d_2}\frac{\mu(d_1d_2)\varrho_3(pd_0d_1d_2)}{pd_0d_1d_2}g(pd_0d_1)g(pd_0d_1).
\end{align*}
The analysis now follows essentially as before. When \cite[Lemma 3]{Zhang} is used to estimate the terms $\mathcal{A}_1(d)$ we can instead use the inequality 
\begin{equation}\mathcal{A}_3(dp)\le \theta_3(p)\mathcal{A}_3(d)+O((\log{D})^{l-2+\epsilon}).\label{eq:A3pInequality}\end{equation}
The only other additional constraint is that $(d,p)=1$, which can be dropped for an upper bound in the final estimations.

This argument then gives
\begin{align}
|\Sigma_{1,p}|+|\Sigma_{2,p}|+|\Sigma_{3,p}|\le \frac{\varrho_3(p)\theta_3(p)^2}{p}\kappa_3\frac{A}{\phi(A)}\binom{2l-2}{l-1}\frac{\mathfrak{S}(\log{D})^{k+2l-1}}{(k+2l-1)!}.
\end{align}
We now sum this bound over $p$ to obtain a total error of
\begin{align}
&\kappa_3\frac{A}{\phi(A)}\binom{2l-2}{l-1}\frac{\mathfrak{S}(\log{D})^{k+2l-1}}{(k+2l-1)!}\sum_{p\le D_1}\frac{\varrho_3(p)\theta_3(p)^2}{p}\frac{2(p-k)\log{p}}{(k+1)p-k}\nonumber\\
&=\kappa_3\frac{A}{\phi(A)}\binom{2l-2}{l-1}\frac{\mathfrak{S}(\log{D})^{k+2l-1}}{(k+2l-1)!}\sum_{p\le D_1}\left(\frac{2\log{p}}{p}+O\left(\frac{\log{p}}{p^2}\right)\right)\nonumber\\
&= (2+o(1))(\log{D_1})\kappa_3\frac{A}{\phi(A)}\binom{2l-2}{l-1}\frac{\mathfrak{S}(\log{D})^{k+2l-1}}{(k+2l-1)!}.
\end{align}
\end{proof}
\begin{lmm}\label{lmm:HardpBound}
Let
\[M_2^*=\sum_{\substack{p\nmid A\\ p\le D_1}}\frac{2(p-k)\log{p}}{(k+1)p-k}\sum_{p|d}\frac{\lambda_d\lambda_e\varrho_3([d,e])}{[d,e]}.\]
Then we have that
\[|M_2-M_2^*|\le 2\varpi\kappa_3\frac{A}{\phi(A)}\binom{2l-2}{l-1}\frac{\mathfrak{S}(\log{D})^{k+2l}}{(k+2l-1)!}(1+o(1)),\]
where $\kappa_3$ is defined in Lemma \ref{lmm:EasyBound}.
\end{lmm}
\begin{proof}
Analogously to \S 4 of \cite{Zhang}, we first bound the difference $|M_2-M_2^*|$ by
\begin{equation}
\sum_{p\nmid A}\frac{2(p-k)\log{p}}{(k+1)p-k}(|\Sigma_{1,p}^*|+|\Sigma_{2,p}^*|+|\Sigma_{3,p}^*|),
\end{equation}
where
\begin{align}
\Sigma_{1,p}^*&=\sum_{\substack{d_0\le x^{1/4}\\ (d_0,p)=1}}\sum_{(d_1,p)=1}\sum_{(d_2,p)=1}\frac{\mu(d_1d_2p)\varrho_3(pd_0d_1d_2)}{pd_0d_1d_2}g(pd_0d_1)g(d_0d_2),\\
\Sigma_{2,p}^*&=\sum_{\substack{d_0\le x^{1/4}\\ (d_0,p)=1\\ d_0|\mathcal{P}}}\sum_{\substack{(d_1,p)=1\\ d_1|\mathcal{P}}}\sum_{\substack{(d_2,p)=1\\ d_2|\mathcal{P}}}\frac{\mu(d_1d_2p)\varrho_3(pd_0d_1d_2)}{pd_0d_1d_2}g(pd_0d_1)g(d_0d_2),\\
\Sigma_{3,p}^*&=\sum_{\substack{x^{1/4}<d_0\le D\\ d_0\nmid \mathcal{P}}}\sum_{(d_1,p)=1}\sum_{(d_2,p)=1}\frac{\mu(d_1d_2p)\varrho_3(pd_0d_1d_2)}{pd_0d_1d_2}g(pd_0d_1)g(d_0d_2).
\end{align}
We first consider $\Sigma_{1,p}$. we wish to put this into a simpler form. Since $\varrho_3$ is supported only on square-free integers, we can insert the conditions $(d_0,d_1)=(d_0,d_2)=(d_1,d_2)=1$. With these conditions we may split up the arguments of $\mu$ and $\rho_3$ due to their multiplicativity. We then rewrite the condition $(d_1,d_2)=1$ using M\"obius inversion. This gives
\begin{align}
\Sigma_{1,p}^*&=\sum_{\substack{d_0\le x^{1/4}\nonumber\\ (d_0,p)=1}}\sum_{(d_1,d_0p)=1}\sum_{(d_2,d_0p)=1}\frac{\mu(d_1)\mu(d_2)\mu(p)\varrho_3(p)\varrho_3(d_0)\varrho_3(d_1)\varrho_3(d_2)}{pd_0d_1d_2}\\
&\qquad\times g(pd_0d_1)g(d_0d_2)\sum_{q_1|d_1,d_2}\mu(q_1)\nonumber\\
&=\frac{-\varrho_3(p)}{p}\sum_{q}\frac{\mu(q)\varrho(q)^2}{q^2}\sum_{d_0\le x^{1/4}}\frac{\varrho_3(d_0)}{d_0}\sum_{(d_1,d_0pq_1)=1}\frac{\varrho_3(d_1)\mu(d_1)}{d_1}g(pd_0d_1q_1)\nonumber\\
&\qquad\times\sum_{(d_2,d_0pq_1)=1}\frac{\mu(d_2)\varrho_3(d_2)}{d_2}g(d_0d_2q_1)\nonumber\\
&=\frac{-\varrho_3(p)}{p}\sum_{q}\frac{\mu(q)\varrho(q)^2}{q^2}\sum_{d_0\le x^{1/4}}\frac{\varrho_3(d_0)}{d_0}\mathcal{A}_3(d_0pq)\sum_{(d_2,d_0pq_1)=1}\frac{\mu(d_2)\varrho_3(d_2)}{d_2}g(d_0d_2q_1).
\end{align}
We rewrite the condition $(d_2,p)=1$ in the inner sum by M\"obius inversion. This gives
\begin{align}
\sum_{(d_2,d_0pq_1)=1}\frac{\mu(d_2)\varrho_3(d_2)}{d_2}g(d_0d_2q_1)&=\sum_{(d_2,d_0q_1)=1}\frac{\mu(d_2)\varrho_3(d_2)}{d_2}g(d_0d_2q_1)\sum_{q_2|p,d_2}\mu(q_2)\nonumber\\
&=\sum_{q_2|p}\frac{\varrho_3(q_2)}{q_2}\sum_{(d_2,d_0q_1q_2)=1}\frac{\mu(d_2)\varrho_3(d_2)}{d_2}g(d_0d_2q_1q_2)\nonumber\\
&=\mathcal{A}_3(d_0q_1)+\frac{\varrho_3(p)}{p}\mathcal{A}_3(d_0q_1p).
\end{align}
Thus we obtain
\begin{align}
\Sigma_{1,p}^*&=\frac{-\varrho_3(p)}{p}\sum_{(d_0,p)=1}\frac{\varrho_3(d_0)}{d_0}\sum_{(q,pd_0)=1}\frac{\mu(q)\varrho_3(q)^2}{q^2}\nonumber\\
&\qquad\times\left(\mathcal{A}_3(d_0pq)\mathcal{A}_3(d_0q)+\frac{\varrho_3(p)}{p}\mathcal{A}_3(d_0pq)^2\right).
\end{align}
Analogously to the argument in \cite{Zhang}, we can restrict the sum over $q$ to $q\le D_0$ at a cost of an error of size $O(D_0^{-1}p^{-1}(\log{D})^B)$ for some constant $B$. Letting $d=d_0q$ then gives
\begin{equation}
\Sigma_{1,p}^*=\frac{-\varrho_3(p)}{p}\sum_{\substack{d\le x^{1/4}D_0\\(d,p)=1}}\frac{\varrho_3(d)\theta_3^*(d)}{d}\left(\mathcal{A}_3(dp)\mathcal{A}_3(d)+\frac{\varrho_3(p)}{p}\mathcal{A}_3(dp)^2\right)+O\left(\frac{(\log{D})^B}{pD_0}\right),
\end{equation}
where
\begin{equation}
\theta_3^*(d)=\sum_{\substack{d_0q=d\\ d_0<x^{1/4}\\ q<D_0}}\frac{\mu(q)\varrho_3(q)}{q}.
\end{equation}
An analogous argument can be applied to $\Sigma_{2,p}^*$ and $\Sigma_{3,p}^*$ which gives
\begin{align*}
\Sigma_{2,p}^*&=\frac{-\varrho_3(p)}{p}\sum_{\substack{d\le x^{1/4}D_0\\ d|\mathcal{P}\\ (d,p)=1}}\frac{\varrho_3(d)\theta_3^*(d)}{d}\left(\mathcal{A}^*_3(dp)\mathcal{A}^*_3(d)+\frac{\varrho_3(p)}{p}\mathcal{A}^*_3(dp)^2\right)+O\left(\frac{(\log{D})^B}{pD_0}\right),\\
\Sigma_{3,p}^*&=\frac{-\varrho_3(p)}{p}\sum_{\substack{d\le x^{1/4}D_0\\ d|\mathcal{P}\\ (d,p)=1}}\frac{\varrho_3(d)\tilde{\theta}_3^*(d)}{d}\left(\mathcal{A}_3(dp)\mathcal{A}_3(d)+\frac{\varrho_3(p)}{p}\mathcal{A}_3(dp)^2\right)+O\left(\frac{(\log{D})^B}{pD_0}\right),
\end{align*}
where
\begin{align}
\mathcal{A}^*_3(d)&=\sum_{\substack{(r,d)=1\\ r|\mathcal{P}}}\frac{\mu(r)\varrho_3(r)g(dr)}{r},\\
\tilde{\theta}(d)&=\sum_{\substack{d_0q=d\\x^{1/4}<d_0}}\frac{\mu(q)\varrho_3(q)}{q}.
\end{align}
The rest of Zhang's argument now essentially follows as before. The differences are, as in Lemma \ref{lmm:EasypBound}, when Zhang uses the asymptotic expression for $\mathcal{A}_1(d)$ we instead use the upper bound from the inequality \eqref{eq:A3pInequality}, and in the final estimations from the sums over $d$ we drop the condition $(d,p)=1$ to obtain an upper bound. This gives us
\begin{align}
|\Sigma_{1,p}|+|\Sigma_{2,p}|+|\Sigma_{3,p}|&\le \kappa_3\frac{A}{\phi(A)}\left(\frac{\varrho_3(p)\theta_3(p)}{p}+\frac{\varrho_3(p)\theta_3(p)^2}{p}\right)\binom{2l-2}{l-1}\nonumber\\
&\qquad\times\frac{\mathfrak{S}(\log{D})^{k+2l-1}}{(k+2l-1)!}(1+o(1)).
\end{align}
We now perform the summation over $p$. We see that
\begin{align}
\sum_{p<D_1}\frac{2(k-p)\log{p}}{(k+1)p-k}\left(\frac{\varrho_3(p)\theta_3(p)}{p}+\frac{\varrho_3(p)\theta_3(p)^2}{p^2}\right)&=\sum_{p\le D_1}\left(\frac{2\log{p}}{p}+O\left(\frac{\log{p}}{p^2}\right)\right)\nonumber\\
&\le (2+o(1))(\log{D}).
\end{align}
This gives us the bound stated in the Lemma.
\end{proof}
\begin{lmm}\label{lmm:MiSizes}
Let $M_1^*,M_2^*$ and $M_3^*$ be defined as in Lemmas \ref{lmm:EasyBound}, \ref{lmm:EasypBound} and \ref{lmm:HardpBound}. We have that
\begin{align*}
M_1^*&=\frac{\mathfrak{S}A}{\phi(A)}\binom{2l-2}{l-1}\frac{1}{(k+2l-1)!}(\log{D})^{k+2l-1}+O\left(\frac{(\log{D})^{k+2l-1}}{\log\log{D}}\right)\\
|M_2^*|&\le\frac{\mathfrak{S}A}{\phi(A)}\binom{2l-2}{l-1}\frac{2(l-1)}{l(k+2l)!}(\log{D})^{k+2l}+O\left(\frac{(\log{D})^{k+2l}}{\log\log{D}}\right)\\
|M_3^*|&\le\frac{\mathfrak{S}A}{\phi(A)}\binom{2l-2}{l-1}\frac{2}{(k+2l)!}(\log{D})^{k+2l}+O\left(\frac{(\log{D})^{k+2l}}{\log\log{D}}\right)
\end{align*}
\end{lmm}
\begin{proof}
This follows from the estimation of the equivalent terms `$M_{2,1},M_{2,2},M_{2,3}$', adapted to our notation, which is performed in \cite[Pages 40-44]{HoTsang}. We note that our sieve weights differ from those used in \cite{HoTsang} only by a constant factor of $(\log{D})^{k+l}/(k+l)!$.

The only difference in the argument is that we have the additional restriction that $p\le D_1$ in the terms $M_2^*$ and $M_3^*$. However, at the point in the argument when the sum over $p$ is evaluated, we may drop this requirement to obtain a bound rather than an asymptotic estimate. Since all further estimations are over terms of the same signs, these bounds correspondingly produce a upper bounds for $|M_2^*|$ and $|M_3^*|$. With further effort one can asymptotically evaluate the terms $M_2^*$ and $M_3^*$, but the loss in our argument here is comparable to the size of $\kappa_3$, which will be small.
\end{proof}
\section{Proof of Proposition \ref{prpstn:MainProp}}
We can now complete the proof of Proposition \ref{prpstn:MainProp}.

The first statement which bounds $S'$ is follows from the argument of \cite[Page 45]{HoTsang}. The result is larger by a factor $(\log{D})^{k+2l}$ since each of our $\lambda_d$ are larger by a constant factor of $(\log{D})^{k+l}/(k+l)!$.

The second and third statements which bound $S_1$ and $S_2(L)$ are the equivalent statements to the bounds \cite[Inequalities (4.20) and (5.6)]{Zhang}. We note that in Zhang's work the linear equations are of the form $L_i(n)=n+h_i$ rather than $An+a_0+h_i$. This essentially leaves the proof of the result for $S_1$ unchanged, but causes a very minor change in the proof of the bound $S_2$. We have
\begin{align}
S_2(L_j)&=\sum_{d,e}\lambda_d\lambda_e\sum_{\substack{N\le n< 2N\\ [d,e]|\Pi(n)}}\theta(L_j(n))\nonumber\\
&=\frac{(\pi(2AN)-\pi(AN))}{\phi(A)}\sum_{d,e}\frac{\lambda_d\lambda_e\varrho_2([d,e])}{\phi([d,e])}+O(E_j)+O(N^\epsilon)\nonumber\\
&=\frac{AN(1+o(1))}{\phi(A)\log{N}}\sum_{d,e}\frac{\lambda_d\lambda_e\varrho_2([d,e])}{\phi([d,e])}+O(E_j)+O(N^\epsilon).\label{eq:S2Expansion}
\end{align}
where $\varrho_2$ is the multiplicative function defined on square-free integers with
\begin{align}
\varrho_2(p)&=\begin{cases}k-1,\qquad&p\nmid A,\\0,&\text{otherwise,}\end{cases}\\
E_j&=\sum_{\substack{d<D^2\\ d|\mathcal{P}, (d,A)=1}}\tau_3(d)\varrho_2(d)\sum_{c\in C_j^*(d)}|\Delta(\theta;Ad,c)|,\\
\Delta(\theta;d,c)&=\sum_{\substack{AN\le n<2AN\\ n\equiv c\pmod{d}}}\theta(n)-\frac{1}{\phi(d)}\sum_{AN\le n<2AN}\theta(n),\\
\mathcal{C}_j^*(d)&=\Bigl\{c:1\le c\le Ad,(c,d)=1,c\equiv h_j+a_0\pmod{A},\nonumber\\
&\qquad\prod_{i=1}^{k}(c-h_j+h_i)\equiv 0\pmod{d}\Bigr\}.
\end{align}
Since $A=O(1)$, $D_1=N^\varpi/A$ and $D=N^{1/4+\varpi}/A$, Zhang's Theorem 2 now bounds $E_j$ by essentially the same argument. We see that, by the Chinese remainder theorem, there is a bijection $\mathcal{C}^*_i(qr)\rightarrow \mathcal{C}^*_i(q)\times\mathcal{C}^*_i(r)$ when $|\mu(qrA)|=1$, which gives the relevant equivalent of \cite[Lemma 5]{Zhang}. The only other change required is a trivial adjustment to the terms in the argument following Zhang's inequality (10.6) to take into account the additional congrunce restriction $c\equiv h_j+a_o\pmod{A}$. 

The rest of the main analysis of $S_2$ goes through correspondingly. The only change is that in Lemmas 2 and 3 we have $\phi(A)\mathfrak{S}/A$ in place of $\mathfrak{S}$. This causes us  to gain a factor $\phi(A)/A$, which cancels with the factor $A/\phi(A)$ which we have in \eqref{eq:S2Expansion}.

The final statement bounding $S_3$ is a consequence of simply combining the results of Lemmas \ref{lmm:DivisorBasic}, \ref{lmm:EasyBound}, \ref{lmm:EasypBound}, \ref{lmm:HardpBound} and \ref{lmm:MiSizes}.
\bibliographystyle{acm}
\bibliography{bibliography}
\end{document}